\documentclass[12pt]{amsart}
 \usepackage{mathrsfs}
\usepackage{graphicx}
\usepackage{amssymb}
\usepackage{amsmath}
\usepackage{amsthm}
\usepackage{amscd}
\usepackage[all]{xy}
\usepackage{dsfont}

\newtheorem{thm}{Theorem}[section]

\newtheorem{lem}[thm]{Lemma}

\newtheorem{prop}[thm]{Proposition}
\theoremstyle{definition}
\newtheorem{defn}[thm]{Definition}

\theoremstyle{remark}
\newtheorem{rem}[thm]{Remark}
\numberwithin{equation}{section}

\DeclareMathOperator{\Spec}{Spec}
\DeclareMathOperator{\Spf}{Spf}

\newcommand{\shf}[1]{\mathcal{#1}}

\newcommand{\ra}{\rightarrow}

\newcommand{\sembrack}[1]{[\![#1]\!]}

\newcommand{\db}{\mathbb}

\makeatletter

\newcommand{\Rmnum}[1]{\expandafter\@slowromancap\romannumeral #1@}
\makeatother

\numberwithin{equation}{section}

\begin{document}

\title[fixed point locus of the Verschiebung]{The 
 fixed point locus of the Verschiebung  on $\shf{M}_X(2, 0)$ for genus-2 curves $X$ in charateristic $2$}
\author{YanHong Yang}
\address{Department of Mathematics, Columbia University,  Room 509, MC 4406,
2990 Broadway
\\ New York, NY 10027, USA}
\curraddr{}
\email{yhyang@math.columbia.edu}
\date{}
\maketitle


\begin{abstract} 
In this note, we prove that for every ordinary genus-$2$ curve $X$ over a finite field $\kappa$ of characteristic $2$ with  $\text{Aut}(X/\kappa)=\db{Z}/2\db{Z} \times S_3$, there exist $\text{SL}(2,\kappa\sembrack{s})$-representations of $\pi_1(X)$ such that the image of $\pi_1(\overline{X})$ is infinite. This result gives a geometric interpretation of Laszlo's counterexample \cite{Laszlo}  to a question regarding the finiteness of the geometric monodromy of representations of the fundamental group \cite{de Jong}.
\end{abstract}

\section{Introduction}
 It was conjectured by de Jong in \cite[Conjecture 2.3]{de Jong} that  given a finite field $\db{F}$ of characteristic $l$ and a  normal variety $Y$ over a finite field $\kappa$ of characteristic $p\neq l$,  every representation $\rho: \pi_1(Y)\ra \text{GL}(r, \db{F}(\!(s)\!))$ has a finite geometric monodromy. This conjecture was proved by de Jong in the $\text{GL}_2$-case 
\cite{de Jong}, by B\"ockle--Khare in the $\text{GL}_n$-case under some mild condition \cite{Bockle-Khare} and by Gaitsgory modulo the theory of $\db{F}(\!(s)\!)$-sheaves \cite{Gaitsgory}.  Then a natural question comes up: if the hypothesis $l\neq p$ is dropped and moreover $Y$ is proper over $\kappa$, does the conjecture remain true?  Note that when $Y/\kappa$ is not proper, a counterexample has already been given in \cite{de Jong}

In \cite{Laszlo}, Laszlo gave a negative answer to  the above question. He showed that there exists a non-trivial family of rank-$2$  bundles fixed by the square of Frobenius over a specific genus-$2$ curve $C_0/\db{F}_2$. From this he deduced the existence of the desired representations of $\pi_1(C_0\otimes\db{F}_{2^d})$. Recently, Esnault and Langer \cite{Esnault-Langer} have employed Laszlo's example to improve the statement of a $p$-curvature conjecture in characteristic $p$.

It is suspected by de Jong that the representations with an infinite geometric monodromy are rare. Thus one would like to understand the underlying mechanics of Laszlo's example and to know whether such representations exist in other characteristic.

In this note, we give a geometric interpretation of Laszlo's example based on the study of the action of the automorphism group of the curve; this interpretation allows us to produce a family of similar examples. Meanwhile, our method also provides some indication regarding higher characteristic, though it does not directly allow us to produce examples. 

Now we give a brief summary of our results. In \cite{Laszlo}, Laszlo deduced  representations   from a non-trivial family of bundles. We show that the converse also holds.  This equivalence is of course well-known to the experts.
\begin{thm}\label{poslocus}
Let $Y$ be a projective smooth geometrically connected curve over a finite field $\kappa$ and $\shf{M}_{Y}(r, 0)$ be the coarse moduli space of rank-$r$ semistable bundles over $\overline{Y}$ with trivialized determinant. Denote by $\xymatrix{V: \shf{M}_Y(r, 0)  \ar@{.>}[r] & \shf{M}_Y(r, 0)}$ the rational map defined by $[E]\mapsto [F^*_{geo}E]$  with respect to the geometric Frobenius map $F_{geo}$ of $Y$ over $\kappa$. Then the following are equivalent:\\
(1)There exists a finite extension $\widetilde{\kappa}$ of $\kappa$ and a representation  $\rho:\pi_1(Y\otimes_\kappa\widetilde{\kappa}) \ra \text{SL}(r, \widetilde{\kappa}\sembrack{s} )$ such that  $\rho|_{\pi_1(\overline{Y})}\text{ mod }s$ is absolutely irreducible and $\#\rho(\pi_1(\overline{Y}))=\infty$. \\
(2) There exists some $N\in\db{N}$ such that the fixed (geometric)  point locus $ \text{Fix}(V^N)=\{x\in\shf{M}_{Y}(r, 0)|V^N(x)=x \}$  is of positive dimension and contains a stable point in a connected component.  
\end{thm}

Any power of $V$ is  also called  Verschiebung by abuse of notations.  Because of the above equivalence, the question of looking for representations is converted to studying the fixed point locus of the Verschiebung. 
In \cite{Laszlo}, the expression of the Verschiebung $V_{C_0}$ for  $C_0$ was applied to locate a projective line $\bigtriangleup$ in $\shf{M}_{C_0}(2, 0)$ such that  $(V_{C_0}^2)|_{\bigtriangleup}$ is the identity map.  Here our observation is that $\bigtriangleup$ is the fixed point locus of the $G$-action on $\shf{M}_{C_0}(2, 0)$, where $G=\text{Aut}(C_0\otimes\db{F}_{2^2}/\db{F}_{2^2})=\db{Z}/2\db{Z} \times S_3$. Indeed, this property is  common to all genus-$2$ ordinary curves in characteristic $2$ with a $G$-action. Note that if a scheme $Y$ is defined over a finite field $\kappa$ with order $p^d$, we will call the $d^{th}$-power of the absolute Frobenius map of $Y$ as the geometric Frobenius map of $Y$ over $\kappa$.
\begin{thm}\label{fixline}
Let $X$ be a projective smooth ordinary curve of genus $2$ over a finite field $\kappa$ of characteristic $2$ with  $\text{Aut}(X/\kappa)= G$.  Let $\xymatrix{V: \shf{M}_X(2, 0)  \ar@{.>}[r] & \shf{M}_X(2, 0)}$ be the rational map defined by taking pullback of bundles with respect to the geometric Frobenius map of $X$ over $\kappa$. Then the fixed point locus of the $G$-action on $\shf{M}_{X}(2, 0)$ is a projective line, denoted by $\bigtriangleup_X$. And  $V|_{\bigtriangleup_X}=\text{id}_{\bigtriangleup_X}$. 
\end{thm}

Combining with Theorem \ref{poslocus}, for every curve in Theorem \ref{fixline}, there exist  representations of the fundamental group with an infinite geometric monodromy. 

A large part of the proof of Theorem \ref{fixline}  can be applied to higher characteristic, particularly the application of a group action in locating a sublocus in the moduli space. However, when considering whether the restriction of the Verschiebung to the sublocus is reduced to a linear map, the condition regarding the existence of a single base point on the sublocus is sufficient only in charateristic $2$. In other characteristic, more is required to ensure that the restriction of the Verschiebung is the identity.

The note is organized as follows. In Section \ref{representation}, we establish equivalences among different categories under concern. In Section \ref{moduli}, we prove Theorem \ref{fixline}.   In Section \ref{highchar}, we discuss the case of characteristic $p>2$.

\textbf{Acknowledgment:} I would like to thank my adviser Aise Johan de Jong for introducing me the topic and for sharing many ideas and having many enlightening discussions with me. I also would like to thank Professor H\'el\`ene Esnault and Professor Adrian Langer for pointing out several confusing places in an earlier version of this note  and express my gratitude to Professor Deligne for some helpful suggestions.

\section{Representations and Frobenius-periodic vector bundles}
\label{representation}
 In this section, we explain how representations are related to families of bundles. Indeed, there are  equivalences among the categories of Frobenius-periodic vector bundles, smooth \'etale sheaves and representations. 

Notation: $\kappa$ is a finite field of order $q=p^d$,   $S=\Spec\kappa\sembrack{s}$, $\shf{S}=\Spf \kappa\sembrack{s}$ and $S_n= \Spec\kappa\sembrack{s}/(s^n)$ for $n\in\db{Z}^+$;  $Y$ is a noetherian $\kappa$-scheme and $F_Y$ is the absolute Frobenius of $Y$. By ``vector bundle", we mean a locally free sheaf of finite rank.

\subsection{}  
Our definition of smooth \'etale $\kappa\sembrack{s}$-sheaves is similar to that of  a lisse $l$-adic sheaf in \cite[Chap.V, \S 1]{Milne}. When $Y$ is connected,  there is  an equivalence  between the category of locally free smooth $\kappa\sembrack{s}$-sheaves over $Y_{et}$ and the category of continuous $\pi_1(Y)$-modules that are free $\kappa\sembrack{s}$-modules of finite rank, denoted by
$\shf{C}_{1, Y_{et}}\rightleftharpoons \shf{C}_{2, \pi_1(Y)}$.


\begin{defn}\label{frovb-0}
 A vector bundle $\shf{F}$ over $Y\times_{\kappa} S$ (resp. $Y\times_{\kappa} S_n$) is said to be  \textit{Frobenius-periodic}  if $\exists$ an isomorphism $\xi:\shf{F} \ra (F_Y^d\times\text{id}_S)^*\shf{F}$ (resp. $\xi:  \shf{F}\ra (F_Y^d\times\text{id}_{S_n})^*\shf{F}$), denoted by $(\shf{F}, \xi)$.  
A \textit{Frobenius-periodic vector bundle}   over $Y\times_{\kappa} \shf{S}$ is a projective system $(\shf{F}, \xi)=((\shf{F}_n, \xi_n))_{n\in\db{Z}^+}$ of sheaves over $|Y_{zar}|$ such that for each $n$, $\shf{F}_n$ is a Frobenius-periodic vector bundle  over $Y\times_{\kappa} S_n$, the given map $\shf{F}_{n+1}\ra\shf{F}_n$ is compatible with $\xi_n$'s and is isomorphic to the natural map $\shf{F}_{n+1}\ra \shf{F}_{n+1}\otimes_{\kappa\sembrack{s}}\kappa\sembrack{s}/(s^n)$.  
\end{defn}


\begin{defn}\label{trvbdl}
Given $(\shf{F}_n, \xi_n)$ over $Y\times_{\kappa} S_n$. For any morphism $Z\overset{f}\ra Y$, $((f\times\text{id}_{S_n})^*\shf{F}_n, (f\times\text{id}_{S_n})^*\xi_n)$ can be viewed as a Frobenius-periodic vector bundle over $Z\times_\kappa S_n$, denoted by $f^*(\shf{F}_n, \xi_n)$ or $(f^*\shf{F}_n, f^*\xi_n)$.\\
A section $s\in\Gamma(Y\times_\kappa S_n, \shf{F}_n)$ is said to be \textit{fixed by $\xi$} if $\xi_n(s)=(F_Y^d\times\text{id}_{S_n})^*s\doteq 1\otimes s$; $(\shf{F}_n, \xi_n)$ is said to be \textit{trivializable} if $\shf{F}_n$ has a global basis fixed by $\xi_n$; it is said to be \textit{\'etale trivializable} if $\exists: Y_n\overset{f}\ra Y$ a finite \'etale morphism such that $f^*(\shf{F}_n, \xi_n)$ is trivializable, in this case we also say that $Y_n/Y$ trivializes $(\shf{F}_n, \xi_n)$. 
\end{defn}

\begin{rem}\label{bdl-shf}
Given $(\shf{F}_n, \xi_n)$ over $Y\times_{\kappa} S_n$, an \'etale sheaf can be defined as follows:
\[(U\overset{f}\ra Y)\in\text{Et}(Y)\mapsto  \{s\in \Gamma(U, f^*\shf{F})| f^*(\xi)(s)=1\otimes s \}.
\]
 We will see that it is a locally free smooth $\kappa\sembrack{s}/(s^n)$-sheaf in Lemma \ref{etale-triv} .\end{rem}

Recall from \cite[Appendix I]{Freitag-Kiehl} that a covering space of $Y$ is a finite \'etale morphism $f: Z\ra Y$ and it is Galois if  $\#\text{Aut}(Z/Y)=\text{deg}(f)$. 

\begin{lem}\label{etale-triv}
  Given $(\shf{F}, \xi)=((\shf{F}_n, \xi_n))_{n\in\db{Z}^+}$ over $Y\times_{\kappa} \shf{S}$. Then there exists a family of covering spaces $Y\leftarrow Y_1\leftarrow Y_2\leftarrow \cdots \leftarrow Y_n\leftarrow\cdots$ such that $Y_n/Y$ trivializes $(\shf{F}_n, \xi_n)$. 
\end{lem}

\begin{proof} Prove by induction on $n$. 
Case $n=1$ is exactly proved in \cite[Proposition 1.2]{Lange-Stuhler}.\\
Induction step: assume that there is a covering space $Y_n\ra Y$ that factors through $Y_{n-1}$   and trivializes $(\shf{F}_n, \xi_n)$.
Let  $\{e^n_1, ... , e^n_r\}$ be a basis of
 $\shf{F}_{n+1}|_{U\times_\kappa S_{n+1}}$ for an affine open subscheme $U\subset Y_n$ s.t. it extends to a global basis of $\shf{F}_n|_{Y_n\times_\kappa S_n}$ fixed by $\xi_n$, i.e. 
\[ \xi_{n+1}\{e^n_1, ... , e^n_r\}= \{(F_{Y_n}^d\times\text{id}_{S_{n+1}})^*e^n_1, ... ,(F_{Y_n}^d\times\text{id}_{S_{n+1}})^* e^n_r\} (I_{(r)}+s^nD_n), \]
for some $D_n\in\text{Mat}(r\times r, \shf{O}_{Y_n}(U))$. To find a basis $\{e^{n+1}_1, ... , e^{n+1}_r\}$  fixed by $\xi_{n+1}$ and of the form $ \{e^n_1, ... , e^n_r\} (I_{(r)}+s^n\triangle_{n+1})$,   it is equivalent to find   $\triangle_{n+1}=(m_{ij})$ such that 
\[   D_{n}+\triangle_{n+1}=\triangle_{n+1}^{(q)}, \] 
where $\triangle_{n+1}^{(q)}=(m_{ij}^{q})$. Then define $U_{n+1}=\Spec \shf{O}_{Y_n}(U)[m_{11}, ... , m_{rr} ]/(D_n+\triangle_{n+1}-\triangle_{n+1}^{(q)})$. Clearly $U_{n+1}\ra U$ is a covering space and trivializes $(\shf{F}_{n+1}, \xi_{n+1})$.  Therefore, the \'etale sheaf associated to $(\shf{F}_{n+1}, \xi_{n+1})$ is locally free and smooth, by \cite[Chap.V, \S1]{Milne}, there exists a covering space $Y_{n+1}\ra Y_n\ra Y$ that trivializes $(\shf{F}_{n+1}, \xi_{n+1})$.
\end{proof}

\begin{rem} 
Actually for an affine open covering $\{ U\}$ of $Y_n$, the local covering spaces $\{U_{n+1}\ra U\}$ can be built up canonically to a covering space $Y_{n+1}\ra Y_n$.
\end{rem}

The trivial line bundle  with a non-trivializable Frobenius structure may  be trivialized by field extension. To avoid such cases, we give the following definition.  

\begin{defn} \label{strfro}
  $(\shf{F}, \xi)$    over $Y\times_{\kappa} S$ (resp. $Y\times_{\kappa} S_n$)  is said to be \textit{strictly Frobenius-periodic} if $(\text{det}(\shf{F}), \text{det}(\xi))$ is trivializable, denoted by $(\shf{F}, \xi, \text{det}=1)$;  
 $ ((\shf{F}_n, \xi_n))_{n\in\db{Z}^+}$  over $Y\times_\kappa \shf{S}$ is said to be \textit{strictly Frobenius-periodic} if every $(\shf{F}_n, \xi_n)$ is. 
\end{defn}

\begin{prop}\label{SL-rep}
(1) Let $\shf{C}_{1, Y_{et}}$ be  the category  of locally free smooth $\kappa\sembrack{s}$-sheaves over $Y_{et}$ and $\shf{C}_{3, Y_{zar}}$ be the category of Frobenius-periodic vector bundles over $Y\times_{\kappa} \shf{S}$.  Then there is an equivalence  
$\shf{C}_{1, Y_{et}}\rightleftharpoons \shf{C}_{3, Y_{zar}}$.\\
(2) Asssume that  $Y$ is connected.  Let $\shf{C}^{sl}_{2, \pi_1(Y)}$ be the full subcategory of $\shf{C}_{2, \pi_1(Y)}$ whose objects are SL-representations of $\pi_1(Y)$ and $\shf{C}^{str}_{3, Y_{zar}}$ be the full subcategory of $\shf{C}_{3, Y_{zar}}$ whose objects are strictly Frobenius-periodic vector bundles over $Y\times_\kappa \shf{S}$. Then there is an equivalence 
\[\shf{C}^{sl}_{2, \pi_1(Y)} \rightleftharpoons \shf{C}^{str}_{3, Y_{zar}}, \quad \rho\leftrightarrow (\shf{F}_\rho, \xi_\rho, \text{det}=1)\quad \text{or} \quad \rho_{(\shf{F}, \xi)}\leftrightarrow (\shf{F}, \xi,\text{det}=1).\]
\end{prop}

\begin{proof} 
(1)May assume that $Y$ is connected.
The functor $\shf{C}_{3, Y_{zar}}\ra \shf{C}_{1, Y_{et}}$ is clear from Remark \ref{bdl-shf} and Lemma \ref{etale-triv}. The functor $\shf{C}_{1, Y_{et}} \ra \shf{C}_{3, Y_{zar}}$ is  the composition $\shf{C}_{1, Y_{et}} \ra \shf{C}_{2, \pi_1(Y)} \ra \shf{C}_{3, Y_{zar}}$. It follows from Galois descent theory, see \cite[\S 12, Thereom 1]{Mumford}.\\
(2) We only need to show that $(\shf{F}_n, \xi_n, \text{det}=1)$ induces a SL-representation. Let $Y_n\overset{f_n}\ra Y$ be a Galois covering space  that trivializes $(\shf{F}_n, \xi_n)$ by Lemma \ref{etale-triv}. Then the induced representation is the composition $ \pi_1(Y, \overline{y})\ra\text{Gal}(Y_n/Y)\ra\text{GL}(r, \kappa\sembrack{s}/(s^n) )$; with a basis $\{ e^n_1, \cdots, e^n_r\} $ of $f_n^*\shf{F}_n$ preserved by $f_n^*\xi_n$,  the latter is defined by $g\mapsto M_{g^{-1}}$, where $(g^{-1}\times\text{id}_{S_n})^*\{e^n_1,\cdots, e^n_r \}=\{e^n_1,\cdots, e^n_r \}M_{g^{-1}}$. Thus  $M_{g^{-1}}\in \text{SL}(r,\kappa\sembrack{s}/(s^n) )$.
\end{proof}

\subsection{}\label{note}
Now we turn to the case that $Y$ is a projective smooth geometrically connected scheme over $\kappa$. Let $\overline{Y}=Y\times_{\kappa}\Spec\overline{\kappa}$.   In this case,  the categories of  vector bundles over $Y\times_\kappa \shf{S}$ and over $Y\times_\kappa S$  are equivalent, by Grothendieck's existence theorem.  Given  $(\shf{F}, \xi)$ over $Y\times_\kappa S$.  We say that $(\shf{F}, \xi)$ is   \textit{constant} if it is isomorphic to the pullback of $\shf{F}\text{ mod } s$. We refer to \cite{Huybrechts-Lehn} and \cite{Isaacs} regarding  the definition of  geometrically slope-stable vector bundles and that of an absolutely irreducible representation.  Recall that

\begin{lem} \label{cst rpr} (1) \cite[Lemma 3.15]{de Jong} Let $\rho: H\ra \text{GL}(r, K\sembrack{s})$ be a representation of a finite group $H$, where $K$ is a field. If $\rho_0= \rho\text{ mod }s$ is absolutely irreducible, then $\rho\simeq \rho_0\otimes_KK\sembrack{s}$. 
(2) \cite[Lemma 2.7]{de Jong} 
Let $1\ra \Gamma\ra H\ra\hat{Z}\ra 0$ be an exact sequence of profinite groups. Suppose that $\rho: H\ra \text{SL}(V)$ is a continuous representation such that $\rho|_\Gamma$ is absolutely irreducible. Then $\#\rho(\Gamma)<\infty \Leftrightarrow \#\rho(H)<\infty$.
\end{lem}

\begin{lem}\label{constant}
Given  $(\shf{F}, \xi, \text{det}=1)\leftrightarrow \rho$. If  $(\shf{F}, \xi)$ is constant,  then $\#\rho(\pi_1(Y)) < \infty$.  If $\rho \text{ mod }s$ is  absolutely irreducible, then $(\shf{F}, \xi)$ is constant  $\Leftrightarrow\#\rho(\pi_1(Y)) < \infty$ . 
\end{lem}

\begin{proof}
 It follows from  Lemmas \ref{etale-triv}, \ref{cst rpr}(1) and descent theory.
\end{proof}

\begin{prop}\label{useful bijection}  
Given  $(\shf{F}, \xi, \text{det}=1)\Leftrightarrow \rho$ as in Proposition \ref{SL-rep}. TFAE: 
\begin{align*}
 &\shf{F}\text{ mod }s \text{ is geometrically  slope-stable (g.s.s.) }\\
\Leftrightarrow&  (\rho\text{ mod }s)|_{\pi_1(\overline{Y})} \text{ is   absolutely irreducible (a.i.)}.
\end{align*}
If the above conditions hold, then 
$\shf{F} \text{ is  non-constant} \Leftrightarrow  \#\rho(\pi_1(\overline{Y}))=\infty$.
\end{prop}

\begin{proof}
Let $\shf{F}_0=\shf{F}\text{ mod }s$ and $\rho_0=\rho\text{ mod }s$.
(g.s.s.)$\implies$ (a.i.):  The reducibility of $\rho_0|_{\pi_1(\overline{Y})}\otimes\overline{\kappa}$ implies the existence of a proper subbundle of $\shf{F}_0\otimes \overline{\kappa}$ with slope $0$.\\
(a.i.)$\implies$ (g.s.s.): As $\shf{F}_0$ is \'etale trivialized, it is geometrically slope-semistable. Since a subbundle with slope $0$ of a trivial bundle is trivial, then the existence of a proper subbundle of $\shf{F}_0\otimes \overline{\kappa}$ with slope $0$ implies the reducibility of $\rho_0|_{\pi_1(\overline{Y})}\otimes\overline{\kappa}$.\\
Since the absolute irreducibility of $(\rho\text{ mod }s)|_{\pi_1(\overline{Y})}$ implies the same property for $\rho|_{\pi_1(\overline{Y})} $ and $\rho \text{ mod }s$, then
the second equivalence follows  from  Lemmas \ref{constant} and \ref{cst rpr}(2).
\end{proof}

\begin{proof}[Proof of Theorem \ref{poslocus}]
(1)$\implies$(2): By Proposition \ref{SL-rep}, there exists a strictly Frobenius-periodic rank-$r$ vector bundle $(\shf{F}, \xi, \text{det}=1)$ over $(Y\otimes_\kappa \widetilde{\kappa})\times_{\widetilde{\kappa}} \Spec\widetilde{\kappa}\sembrack{s}$. Locally, $(\shf{F}, \xi, \text{det}=1)$ is defined by transition matrices and linear maps.  Let $A\subset \widetilde{\kappa}\sembrack{s}$ be the finitely generated $\widetilde{\kappa}$-algebra generated by  elements appearing in  the matrices that define $(\shf{F}, \xi, \text{det}=1)$. Clearly, there exists canonically a strictly Frobenius-periodic bundle $(\shf{F}', \xi', \text{det}=1)$  over $(Y\otimes_\kappa \widetilde{\kappa})\times_{\widetilde{\kappa}}\Spec A$ such that its  pullback to $Y\times_\kappa\Spec\widetilde{\kappa}\sembrack{s}$ is 
exactly $(\shf{F}, \xi, \text{det}=1)$. As  $\shf{F}'$ can be viewed as a family of bundles over $\overline{Y}$ fixed by  the geometric Frobenius map of $Y\otimes_\kappa \widetilde{\kappa}$ over $ \widetilde{\kappa}$, i.e.
the $N^{th}$-power of the geometric Frobenius map of $Y$ over $\kappa$, where $N=[\widetilde{\kappa}: \kappa]$. Thus the image of the modular morphism  $\Spec A\ra \shf{M}_{Y}(r, 0)$ is  in $\text{Fix}(V^N)$. By Proposition \ref{useful bijection}, $\shf{F}'$ is a non-constant family and consists mostly of stable bundles, thus $\text{Fix}(V^N)$ has the required properties.\\
(2)$\implies$(1): It follows from the construction of $\shf{M}_{X_t}(2,0)$ as a GIT quotient as shown in \cite[Corollary 3.2 \& Lemma 3.3]{Laszlo}.
\end{proof}

From now on, in order to obtain representations with an infinite geometric monodromy, we turn to study the fixed point locus of the Verschiebung.

\section{Proof of Theorem \ref{fixline}}\label{moduli} In this section, $G=\db{Z}/2\db{Z} \times S_3$.
Let $X$ be a projective smooth ordinary curve of genus $2$ over a field $\kappa$ of characteristic $2$ with  $\text{Aut}(X/\kappa)= G$. Except in the proof of Theorem \ref{fixline}, $\kappa$ can be infinite.  Let $X(1)$ be the scheme deduced from $X$ by the extension of scalars $a \mapsto a^{2}$ and $F_{X/\kappa}: X\ra X(1)$ be the relative Frobenius map. Note that the $G$-action on $X$ induces a $G$-action on $X(1)$ that is compatible with  $F_{X/\kappa}$. 

\subsection{$G$-action}\label{grpact}
 In this subsection, we study the fixed point locus of the $G$-action on the Kummer surface $\text{Km}_X$ of X and on the coarse moduli space $\shf{M}_X(2,0)$ of rank-$2$ semistable bundles with trivialized determinant over $\overline{X}=X\otimes_\kappa\overline{\kappa}$. 

Let $\pi_X: X\ra|\omega_{X/\kappa}|=\db{P}^1$ be the canonical morphism of $X$. As $X$ is ordinary, the double covering $\pi_X$ has three branch points. May assume that the points are $\{0, 1, \infty \}$. The $\db{Z}/2\db{Z}$-action on $X$ is generated by the hyperelliptic involution of $\pi_X$, denoted by  $\iota$; the $S_3$-action on $X$ induces an action on the canonical linear system $|\omega_{X/\kappa}|$ and hence can be identified as the permutation group of the branch points $\{0, 1, \infty \}$. Let $\tau_{01}=(01)(\infty)$ and  $\sigma=(01\infty)$.  Note that $\sigma$ fixes four points on $\overline{X}$.

The $G$-action on $X$ induces a $G$-action on the Jacobian $J_{X}$ and thus  on the Kummer surface $\text{Km}_{X}$ of $X$. We can actually figure out the fixed points of $G$ on $\text{Km}_{X}$.

\begin{lem}\label{fixptKm}
The set of the fixed points $(\text{Km}_{X})^G$ of the $G$-action on $\text{Km}_{X}$ consists of  three points: $ \shf{O}_X^{\oplus 2}$,   $ E_{1,X}=\shf{O}_X(Q-\tau_{01}(Q))\oplus \shf{O}_X(\tau_{01}(Q)-Q)$, $ E_{2, X}=\shf{O}_X(Q-\iota\circ\tau_{01}(Q))\oplus\shf{O}_X(\iota\circ\tau_{01}(Q)-Q)$, where $Q\in\overline{X}$ is a fixed point of $\sigma$.
\end{lem}
\begin{proof}
It suffices to find all line bundles $\shf{L}\in \text{Pic}^0(\overline{X})$ such that $g^*\shf{L}\simeq \shf{L} \text{ or } \shf{L}^{-1}$ for $g=\tau_{01}, \tau_{0\infty} \text{ and } \sigma$. As $h^0(\shf{L}\otimes\omega_{X/\kappa})\geq1$, then $\shf{L}\simeq \shf{O}_X(Q_1-Q_2)$ for $Q_1, Q_2\in \overline{X}$. The lemma is proved by a  case-by-case analysis according to the three types of points:  (I) the three ramification points; (II) the four fixed points of $\sigma$; (III) all the others.
\end{proof}

Clearly $E_{1, X}$ and $E_{2,X}$ are indenpendent of the choice of the fixed point of $\sigma$. Take a fixed point $Q_1$ of $\sigma$ on $\overline{X}(1)$, we similarly define $E_{1, X(1)}$ and $E_{2,X(1)}$. We have

 \begin{lem} \label{fixfro}
For $j=1, 2$,  $F_{X/\kappa}^*E_{j, X(1)}=E_{j, X}$. 
\end{lem}

\begin{proof}
 Let $Q=F_{X/\kappa}^{-1}(Q_1)$. Then $Q$ is  a fixed point of $\sigma$ because $Q_1\in \overline{X}(1)$ is  a fixed point of $\sigma$ and $F_{X/\kappa}$ preserves the $G$-action. Since $F_{X/\kappa}^*[\shf{O}_{X(1)}(Q_1-\tau_{01}(Q_1))]=\shf{O}_X(2Q-2\tau_{01}(Q))$ and $F_{X/\kappa}^*[\shf{O}_{X(1)}(Q_1-\iota\circ\tau_{01}(Q_1))]=\shf{O}_X(2Q-2\iota\circ\tau_{01}(Q))$,  it suffices to prove that $3Q\sim 3\tau_{01}(Q)\sim 3\iota\circ\tau_{01}(Q)$. This follows from that $Q, \tau_{01}(Q), \iota(Q), \iota\circ\tau_{01}(Q) $ are the ramification points with index $3$ of  the  quotient $X \ra X/\langle \sigma\rangle\simeq\db{P}^1$ . 
\end{proof}

By (\cite{Laszlo-Pauly}), $\shf{M}_{X}(2, 0)$ is isomorphic to $|2\Theta|\simeq\db{P}^3_\kappa$ and the Kummer surface $\text{Km}_X$ is  a quartic hypersurface.  To find the fixed point locus $(\shf{M}_{X}(2, 0))^G$, we need the following. 

\begin{lem}\label{ptrick}
 Let $H$ be a subgroup of $\text{Aut}(\db{P}^n_{k}/k)$, where $k$ is a field of characteristic $p$. Assume that  $H$ is generated by elements with order of the form $p^r$.  Let $P_1, P_2\in \db{P}^n_{k}$ be fixed by $H$, then the projective line $\overline{P_1P_2}$ is fixed by $H$. 
\end{lem}

\begin{proof}
Identify points $P_1,P_2$  with  vectors $v_1, v_2\in k^{n+1}$. Let $h\in H$ have order $p^r$ and $\widetilde{h}\in\text{GL}(n+1, k)$ be a preimage of $h$. Then $\widetilde{h}^{p^r}=\mu I_{n+1}$. By assumption, $\widetilde{h}(v_1)=\mu_1v_1$ and $\widetilde{h}(v_2)=\mu_2v_2$. Thus $\mu_1^{p^r}=\mu_2^{p^r}\implies \mu_1=\mu_2$. Therefore,  $h$ fixes the line $\overline{P_1P_2}$.
\end{proof}

\begin{prop}\label{proline}
The fixed point locus of the $G$-action on $\shf{M}_{X}(2, 0)$ is a projective line, denoted by  $\bigtriangleup_{X}$.
\end{prop}

\begin{proof}
 As $(\shf{M}_{X}(2, 0))^G\cap \text{Km}_X$ is a set of three points and $\text{Km}_X$ is a hypersurface, by Lemma \ref{proline}, $(\shf{M}_{X}(2, 0))^G$ is a projective line.
\end{proof}

\subsection{Verschiebung} \label{frotwst}
Let $X(n)$ be the scheme deduced from $X$ by the extension of scalars $a \mapsto a^{2^n}$. Denote by $F_n$ the relative Frobenius $F_n: X(n)\ra X(n+1)$ and  by $V_n$ the Verschiebung $ \xymatrix{V_n: \shf{M}_{X(n+1)}(2, 0)  \ar@{.>}[r] & \shf{M}_{X(n)}(2, 0)}, [E]\mapsto [F_n^*E]$.

As $X(n)$ has the same properties as $X$, the results for $X$ also hold for $X(n)$. Let $\bigtriangleup_{X(n)}$ be the projective line of $\shf{M}_{X(n)}(2, 0)$ in Proposition \ref{proline}. As $F_n$ is compatible with the $G$-action, the pullback of a $G$-bundle is a $G$-bundle, hence $V_n(\bigtriangleup_{X(n+1)})\subset\bigtriangleup_{X(n)}$.   

To reduce $V_n|_{\bigtriangleup_{X(n+1)}}: \bigtriangleup_{X(n+1)} \ra \bigtriangleup_{X(n)}$ to  a linear map, we point out a base point of $V_n$ on $\bigtriangleup_{X(n+1)}$. Recall from \cite{Raynaud} that there is a theta characteristic $B_n$ of $X(n)$ defined as $0\ra \shf{O}_{X(n)}\ra F_{n-1*}\shf{O}_{X(n-1)} \ra B_n\ra 0$. Consider the rank-$2$ bundle $F_{n*}B_{n}^{-1}$ over $X(n+1)$. Clearly $\text{det}(F_{n*}B_{n}^{-1})=\shf{O}_{X(n+1)}$. Because
$ 0\ra B_n \ra F_{n}^*(F_{n*}B_{n}^{-1}) \ra B_n^{-1}\ra 0$ and $\text{deg}(B_n)=1$, thus $F_{n*}B_{n}^{-1}$ is stable and  $ F_{n}^*(F_{n*}B_{n}^{-1})$ is unstable. Moreover, $F_{n*}B_{n}^{-1}$ has a $G$-action by construction, thus $[F_{n*}B_{n}^{-1}]\in\bigtriangleup_{X(n+1)}$.  Therefore, 
\begin{lem}\label{linear}
The restriction $ \xymatrix{V_n|_{\bigtriangleup_{X(n+1)}}:\bigtriangleup_{X(n+1)} \ar@{.>}[r] & \bigtriangleup_{X(n)}}$ is a linear map.
\end{lem}

\begin{proof}
By \cite[Proposition 6.1]{Laszlo-Pauly}, $V_n$ is defined by qradratic polynomials.  Thus $V_n|_{\bigtriangleup_{X(n+1)}}$ is given by two quadratic polynomials $\{ h_1, h_2\}$ in two variables. As $V_n|_{\bigtriangleup_{X(n+1)}}$ has a base point $[F_{n*}B_{n}^{-1}]$, $h_1$ and $h_2$ have a common linear factor, thus $V_n|_{\bigtriangleup_{X(n+1)}}$ is reduced to a linear map.
\end{proof}

\begin{proof}[Proof of Theorem \ref{fixline}]  Assume that $\#\kappa=2^d$. Note that $X(d)=X$. As the rational map $V$ is the composition $V_0\circ V_1\circ\cdots\circ V_{d-1}$ and $ \xymatrix{V_n|_{\bigtriangleup_{X(n+1)}}:\bigtriangleup_{X(n+1)} \ar@{.>}[r] & \bigtriangleup_{X(n)}}$ is linear for $0\leq n\leq d-1$ by Lemma \ref{linear}, thus $V|_{\bigtriangleup_{X}}$ is linear. Moreover, Lemma \ref{fixfro} holds for every $X(n)$, i.e. there are  semistable bundles $E_{1, X(n)}, E_{2, X(n)}$  such that  $V_n([E_{j, X(n+1)}])=[E_{j, X(n)}]$ for $j=1,2$.
 Thus $V|_{\bigtriangleup_{X}}$ has three distinct fixed points, i.e. $[\shf{O}_{X}^{\oplus 2}]$, $[E_{1, X}]$ and $[E_{2, X}]$. In conclusion, $V|_{\bigtriangleup_{X}}$ is the identity map.
\end{proof}

\begin{rem}
Actually it can be shown that there exists a vector bundle $\shf{E}$ over $X\times_\kappa \Lambda $ with $\Lambda=\Spec\kappa[\lambda]$ such that (1) the modular morphism $i:\Lambda \ra \bigtriangleup_{X}$ is an open immersion with the only missing point to be $[\shf{O}_X^{\oplus 2}]$; (2) there exists a morphism $(F_X^d\times\text{id}_\Lambda)^*\shf{E}\ra \shf{E}$ which is isomorphic if replacing $\Lambda$ by an open subset; (3) all bundles $\shf{E}_\lambda$ are subbundles of $\omega_{X/\kappa}^2\oplus\omega_{X/\kappa}^2$.
\end{rem}

\begin{rem}
By \cite{Ancochea}, a curve in Theorem \ref{fixline} is defined by the equation
\begin{align}\label{cvequ}
 \quad y^2+(x^2+x)y+(t^2+t)(x^5+x)+t^2x^3=0, \quad t\neq 0, 1.
\end{align}
 Note that the automorphism group $\text{Aut}(C_0/\db{F}_2)$ of  the curve $C_0/\db{F}_2$ in \cite{Laszlo}  is not $ \db{Z}/2\db{Z}\times S_3$. We need to  consider $C_0\otimes_{\db{F}_2}\db{F}_{2^2}$ so that $\text{Aut}(C_0\otimes_{\db{F}_2}\db{F}_{2^2}/\db{F}_{2^2})=\db{Z}/2\db{Z}\times S_3$. That is the case when $t\in\db{F}_4\backslash\db{F}_2$ in Equation (\ref{cvequ}).
\end{rem}

\section{Discussion on higher characteristics}\label{highchar}
In this section, we  discuss some potential generalization of  Theorem \ref{fixline} to  higher characteristic.  

Notations: Let $Y$ be a projective smooth genus-$2$ curve  over a finite field $\kappa$ of characteristic $p>0$.   It is known that $\shf{M}_{Y}(2,0)$  is isomorphic to the linear system $|2\Theta|\simeq \db{P}^3_\kappa$ over $\text{Pic}^1(\overline{Y})$. Let $\xymatrix{V: \shf{M}_{Y}(2, 0)  \ar@{.>}[r] & \shf{M}_Y(2, 0)}$ be the rational map induced by  the geometric Frobenius map of $Y$ over $\kappa$. Let $Y(n)$, $F_{n}: Y(n)\ra Y(n+1)$ and $\xymatrix{V_n: \shf{M}_{Y(n+1)}(2, 0)  \ar@{.>}[r] & \shf{M}_{Y(n)}(2, 0)}$ be the same as in Subsection \ref{frotwst}. By \cite[Proposition A.2]{Laszlo-Pauly1}, $V_n$'s are given by polynomials of degree $p$. Assume that $\#\kappa=p^d$, then $Y(d)=Y$.

Recall the proof of Theorem \ref{fixline}, a large part can be applied to an arbitrary characteristic. In particular, the following two facts are true.

\begin{lem}\label{basept}
If  $V^m|_Z=\text{id}_Z$ for a reduced subscheme $Z\subset \shf{M}_{Y}(2, 0)$ of positive dimension and for some $m>0$, then the  closure $\overline{Z}$ of $Z$ contains a base point of $V_{d-1}$.
\end{lem}

\begin{lem}\label{charp}
Let  $H\subset\text{Aut}(Y/\kappa)$ be a subgroup generated by elements with order of the form $p^n$. Let $\xymatrix{V: \shf{M}_Y(2, 0)  \ar@{.>}[r] & \shf{M}_Y(2, 0)}$ be the rational map given in the notations.  Assume that the set $(\text{Km}_Y)^H$ of fixed points of  $H$  on the Kummer surface $\text{Km}_Y$  is finite. Given a semistable bundle $[E]\in\shf{M}_{Y}(2,0)^H$ satisfying that $F_Y^*E$ is semistable and $[F_Y^*E]\neq [\shf{O}_{Y}^{\oplus 2}]$, where $F_Y$ is  the absolute Frobenius map of $Y$.  Then the fixed point locus of the $H$-action on $\shf{M}_{Y}(2, 0)$ is a projective line, denoted by $\bigtriangleup_Y$. And  the restriction of $V$ to $\bigtriangleup_Y$ is a rational map $V|_{\bigtriangleup_Y}: \db{P}^1_\kappa \ra \db{P}^1_\kappa$.   
\end{lem}

The special property of characteristic $2$ that is used in proving Theorem \ref{fixline} is that the existence of a single base point on $\bigtriangleup_{Y(n+1)}$ is sufficient to lower the degree of the  polynomials that define $(V_n)|_{\bigtriangleup_{Y(n+1)}}$ from $2$ to $1$. Similar cases may happen in other small characteristics. However, in large characteristic, as it is proved in  \cite[Proposition 3.1]{Lange-Pauly} that every $V_n$ has exactly $16$ base points for characteristic $p>2$, then the intersection number of $\bigtriangleup_{Y(n)}$ with  the scheme-theoretic base locus $ \shf{B}_{n}$ of $V_{n-1}$ is required to check if $(V_{n-1})|_{\bigtriangleup_{Y(n)}}$ can be reduced to a linear map. To calculate $\bigtriangleup_{Y(n)}\cap \shf{B}_{n}$,   more  about $V_n$ should be discovered. 

If suitable conditions could be found to ensure that every $(V_n)|_{\bigtriangleup_{Y(n+1)}}$ is linear, then the  map $V|_{\bigtriangleup_Y} $ in Lemma \ref{charp} is linear and non-constant; moreover, as $V|_{\bigtriangleup_Y} $ is defined over a finite field, there exists some $N$ such that $(V|_{\bigtriangleup_Y})^N=\text{id}$. Therefore, by Theorem \ref{poslocus}, we would obtain representations of $\pi_1(Y\otimes_\kappa\widetilde{\kappa})$ with an infinite geometric monodromy.




\begin{thebibliography}{99}
\bibitem{Ancochea}Ancochea G., \textit{Corps hyperelliptiques abstraits de caract\'eristique 2}, Portugaliae Math. 4(1943) 119-128

\bibitem{Biswas-Ducrohet}Biswas  I., Ducrohet L., \textit{An analog of a theorem of Lange and Stuhler for principal bundles}, C. R. Math. Acad. Sci. Paris 345 (2007), 495-497

\bibitem{Bockle-Khare} B\"ockle G., Khare C., \textit{Mod $l$ representations of arithmetic fundamental groups. II. A conjecture of A.J. de Jong}, Compos. Math. 142 (2006), 271-294 

\bibitem{de Jong} de Jong A.J., \textit{A conjecture on arithmetic fundamental groups}, Israel J. Math. 121 (2001), 61-84

\bibitem{Esnault-Langer} Esnault H., Langer A., \textit{On a positive equicharacteristic variant of the p-curvature conjecture}, arXiv: 1108.0103

\bibitem{Freitag-Kiehl} Freitag E., Kiehl R., \textit{\'Etale cohomology and the Weil conjecture}, Ergebnisse der Mathematik und ihrer Grenzgebiete (3), 13. Springer-Verlag, Berlin, 1988

\bibitem{Gaitsgory} Gaitsgory D., \textit{On de Jong's conjecture}, Israel J. Math. 157 (2007), 155-191

\bibitem{Huybrechts-Lehn}Huybrechts D., Lehn M., \textit{The geometry of moduli spaces of sheaves} 2nd edition, Cambridge Mathematical Library. Cambridge University Press, Cambridge, 2010

\bibitem{Isaacs} Isaacs I.M., \textit{Character Theory of Finite Groups}, Academic Press, New York, 1976


\bibitem{Lange-Pauly}Lange H.,  Pauly C., \textit{On Frobenius-destabilized rank-2 vector bundles over curves}, Comment. Math. Helv. 83 (2008), 179-209

\bibitem{Lange-Stuhler} Lange H., Stuhler U., \textit{Vektorb\"undel auf Kurven und Darstellungen der algebraischen Fundamentalgruppe}, Math. Z. 156 (1977), 73-83

\bibitem{Laszlo} Laszlo Y., \textit{A non-trivial family of bundles fixed by the square of Frobenius},  C. R. Math. Acad. Sci. Paris 333 (2001),  651-656 

\bibitem{Laszlo-Pauly}Laszlo Y., Pauly C., \textit{The action of the Frobenius map on rank 2 vector bundles in characteristic 2}, J. Algebraic Geom. 11 (2002), 219-243

\bibitem{Laszlo-Pauly1}Laszlo Y., Pauly C., \textit{The Frobenius map, rank 2 vector bundles and Kummer's quartic surface in characteristic 2 and 3}, Adv. Math. 185 (2004), 246-269

\bibitem{Milne} Milne J.S., \textit{\'Etale cohomology}, Princeton Mathematical Series, 33. Princeton University Press, Princeton, N.J., 1980


\bibitem{Mumford} Mumford D., \textit{Abelian varieties}, Lectures in the Tata Institute of Fundamental Research, Oxford University Press, London, 1974

\bibitem{Raynaud}Raynaud M., \textit{Sections des fibr\'es vectoriels sur une courbe}, Bull. Soc. Math. France 110 (1) (1982), 103-125
\end{thebibliography}
\end{document}